\documentclass[preprint,1p]{elsarticle}
\usepackage{fancyhdr}
\usepackage{lastpage}
\usepackage{color}
\usepackage{fancybox}
\usepackage{moreverb}
\usepackage{listings}
\usepackage{times}
\usepackage[T1]{fontenc}
\usepackage{enumerate}
\usepackage{amsmath,amsfonts,amssymb,latexsym,amsthm,cases}
\usepackage{empheq}
\usepackage{epsf}
\usepackage{tikz}
\usepackage{tkz-graph}
\usepackage{tikz-3dplot}\tdplotsetmaincoords{70}{130}
\usetikzlibrary{calc,decorations.pathreplacing,arrows,positioning,intersections}
\usepackage{algorithmic}
\usepackage{algorithm}
\usepackage{epsfig}
\usepackage{amssymb}
\usepackage{subfigure}
\usepackage{chngpage}
\usepackage{aurical}
\usepackage[english]{babel}
\usepackage{multirow}

\newtheorem{theo}{Theorem}
\newtheorem{lemme}{Lemma}
\newtheorem{cor}{Corollary}

\newtheorem{prop}{Proposition}
\newtheorem{defi}{Definition}

\begin{document}
	
\begin{frontmatter}

\title{On an optimal constraint aggregation method for integer programming and on an expression for the number of integer points in a polytope}
 
\author[ens]{Pierre-Louis Poirion \corref{cor1}}
\ead{pierre-louis.poirion@ensta-paristech.fr}

\author[lix]{Vu Khac Ky}

\author[lix]{Leo Liberti}

\cortext[cor1]{Corresponding author}

\address[ens]{ENSTA ParisTech,  828, Boulevard des Maréchaux F-91762 Palaiseau Cedex (and CEDRIC-CNAM)}

\address[lix]{ CNRS LIX (UMR7161), \'Ecole Polytechnique, 91128 Palaiseau, France}

\begin{abstract}
In this paper we give a new aggregation method for linear Diophantine systems. In particular, we prove that an aggregated system of minimum size can be constructed in polynomial time. We also derive an analytic formula that gives the number of solutions of the system when it is possible to aggregate the system into one equation.
\end{abstract}

\end{frontmatter}

\section{Introduction}
Let us consider a linear Diophantine system:
\begin{equation*}
(F) = \left\{
\begin{array}{lll}
& Ax=b\\
& x \in \mathbb{Z}^n_+
\end{array}
\right. 
\end{equation*}
where $A$ is an $m\times n$ integer matrix and $b\in \mathbb{Z}^m$. Such a system is defined by the intersection of set of equations and the non-negative orthant. An \textit{aggregation} of the system above is a linear Diophantine system having equations that are all linear combinations of the equations in $(F)$. Hence such an aggregated system can be written as: 
\begin{equation*}
(F_T) = \left\{
\begin{array}{lll}
& TAx=Tb\\
& x \in \mathbb{Z}^n_+
\end{array}
\right. 
\end{equation*}
where $T \in \mathbb{Q}^{k \times m}$ is an \textit{aggregation matrix} and where $k \in \mathbb{N}$ defines the \textit{size} of the aggregated system. Obviously, for any $T$, $(F_T)$ is a relaxation of $(F)$, hence every solution of $(F)$ is a solution of $(F_T)$. When both systems have the same set of solutions, we say that $T$ is a \textit{strong} aggregation matrix and that $(F_T)$ is a strong aggregated system. For example,
\begin{equation*}
(\tilde{F}_{(1 \ 2)}) = \left\{
\begin{array}{lll}
& 5x + 4y =9\\
& (x,y) \in \mathbb{Z}^2_+
\end{array}
\right. 
\end{equation*}
is a strong aggregated system for the system
\begin{equation*}
(\tilde{F}) = \left\{
\begin{array}{lll}
& x + 2y =3\\
& 2x +y =3 \\
& (x,y) \in \mathbb{Z}^2_+
\end{array}
\right. 
\end{equation*}
as, in both case, the set of solutions is equal to $\{(1,1)\}$, and we have that \\ $\begin{pmatrix}
1 & 2
\end{pmatrix}\begin{pmatrix}
1 & 2 \\
2 & 1
\end{pmatrix}= \begin{pmatrix}
5 & 4
\end{pmatrix}$. However if we take the aggregation matrix, $T=(1\ 1)$, for $(\tilde{F})$, we obtain the aggregated system:
\begin{equation*}
(\tilde{F}_{(1 \ 1)}) = \left\{
\begin{array}{lll}
& 3x + 3y =6\\
& (x,y) \in \mathbb{Z}^2_+
\end{array}
\right. 
\end{equation*}
which is not a strong aggregated system, as its set of solutions is equal to $\{(1,1),(2,0),(0,2)\}$.\\

 When the aggregated system has the property to have a solution if and only if the original one has a solution, we say that $T$ is a \textit{weak} aggregation matrix. We allow, in the weak aggregation case, the introduction of upper bounds for the variables of the aggregated system, $(F_T)$.
 For example, $(\tilde{F}_{(1 \ 1)})$ is a weak aggregated system, as both set of solutions are non-empty.\\
 
In \cite{bradley1971}, Bradley studies the case of bounded Diophantine systems, where $x$ is bounded above by a vector $u$. He proves that, for a bounded system, it is always possible to construct a strong aggregated system of size one, and hence that we can always transform a bounded integer linear program into an equality constrained knapsack problem. In order to obtain the aggregated problem, the constraints are aggregated two by two, one at a time. In \cite{kendall1977}, Kendall et al. present a refinement of Bradley's method, that reduce the size of the coefficients of the aggregation matrix. In \cite{rosenberg1974}, Rosenberg studies aggregation methods for general Diophantine systems (not necessarily linear) and gives sufficient conditions to construct a strong aggregated system of size $k$. In particular, he proves that, in the linear case, it is always possible to construct a strong aggregated system having a size equal to the dimension of the lineality subspace of the cone generated by $A$ (a precise definition will be given later) plus one. In \cite{kannan1983polynomial}, Kannan proves that when all the entries of the matrix $A$ are positive, it is possible to construct a strong aggregated matrix of size one in polynomial time. The aggregation of a system of inequalities is studied in \cite{chvatal1977}.

In this paper, we will study the minimum size of a strong aggregation matrix $T$ for $(F)$, and prove that such a matrix can be constructed in polynomial time. In particular, we will prove that the bound found by Rosenberg, in \cite{rosenberg1974}, for the size of an aggregated system, in the linear case, is optimal and that a corresponding aggregation matrix can be constructed in polynomial time. This also generalizes the result established by Kannan, in \cite{kannan1983polynomial}, as it easy to see than when all the entries of $A$ are positive, then the cone generated by $A$ is pointed. Our approach is geometric and comes from the following observation proved in \cite{vu2015gaussian}: given a hyperplane $H=\{y\in \mathbb{R}^m\,|\, Ty=0\}$ where the entries of $T$ are independent random variables having a continuous density function, then, almost surely, $T$ is a strong aggregation matrix for $(F)$. This is because, with probability one, $H$ will avoid all the non-zero integer points in $\mathbb{R}^m$. In this paper, we aim to build, in polynomial time, a rational subspace $H\subset \mathbb{R}^m$ that verifies this property. \\

In Section~\ref{sec:1}, we revisit the bounded case and give a new method to build a strong aggregation matrix of size one for the problem. In Section~\ref{sec:2} we study the general case and prove that a strong aggregation matrix of minimum size can be computed in polynomial time. Then, in Section~\ref{sec:weakaggreg} we briefly explain how to build a weak aggregated system of size one for any linear Diopantine system $(F)$. Finally, in Section~\ref{sec:counting}, we explain how to use strong aggregated system to derive an analytic formula that counts the number of solutions for $(F)$.

\section{Notation}
We introduce the following notations. 
For any matrix $\mathcal{M}$, we denote by $\mathcal{M}_i$ the $i$th column of $\mathcal{M}$. We denote by $\mathcal{C}(\mathcal{M})=\{\mathcal{M}x\, |\, x\ge 0\}$, the cone generated by the column of $\mathcal{M}$. Similarly, $\mathcal{L}(\mathcal{M})=\{\mathcal{M}x\, |\, x \mbox{ integer} \}$ denotes the integer lattice generated by $\mathcal{M}$. Furthermore $\mathcal{L}^+(\mathcal{M})=\{\mathcal{M}x\, |\, x \mbox{ non-negative integer} \}$ denotes the non-negative lattice generated by $\mathcal{M}$. \\
We denote by $\|\|_\infty$ the infinity norm; for any matrix $\mathcal{M}$, we denote by $\|\mathcal{M}\|_\infty=\max\limits_{1\le i \le m} \sum \limits_{j=1}^n|\mathcal{M}_{ij}|$ the induced infinity norm on $\mathcal{M}$.
For all $i\in \mathbb{N}$, we denote by $p_i$ the $i$th prime number.  The aggregation matrix will always be denoted by $T \in \mathbb{Q}^{k \times m}$, where $k$ is the size of the aggregated system.

\section{Bounded integer linear problem}\label{sec:1}
In this section, we consider the case of bounded integer linear feasibility problem (BILFP). We will prove that we can build, in polynomial time, a strong aggregation matrix $T$ of size one. Let us consider the following BIFLP, $(F)$, and its aggregated version, $(F_T)$:
\begin{equation*}
(F^B) = \left\{
\begin{array}{lll}
& Ax=b\\
& 0\leq x \leq u \\
& x \in \mathbb{Z}^n
\end{array}
\right. 
\end{equation*}

\begin{equation*}
(F^B_T) = \left\{
\begin{array}{lll}
& TAx=Tb\\
& 0\leq x \leq u \\
& x \in \mathbb{Z}^n
\end{array}
\right. 
\end{equation*}

\noindent Where $A\in \mathbb{Z}^{m \times n},b \in \mathbb{Z}^m,  u \in \mathbb{Z}^n, T\in \mathbb{Z}^{k \times m}$ and $\mbox{rank}(A)=m<n$. Let us denote by $\mathcal{L}(A)$ the lattice generated by $A$. We define the bounded lattice $L_1=\{Ax\ |\, x \in \mathbb{Z}^n,\, 0\leq x \leq u\}$. The feasibility problem $(F^B)$ is therefore equivalent to the set membership problem: ``decide whether $b \in L_1$ or not".\\

\begin{defi}
	Let $H$ be a hyperplane of $\mathbb{R}^m$ and let $S\subset \mathbb{Z}^m$. We say that $H$ avoids $S$ if $H\cap S \subseteq \{0\}$. 
\end{defi}

The following proposition gives an equivalence between the hyperplanes avoiding the set $S=-b+L_1$ and the strong aggregation matrices $T$ for $(F^B)$.
\begin{prop}\label{prop:avoids}
	Let $H=\{y\in\mathbb{R}^m\ |\ Ty=0 \}$, with $T$ a $1\times m$ matrix, be a hyperplane that avoids the set $S=-b+L_1$ then $T$ is a strong aggregation matrix for $(F^B)$.
\end{prop}
\begin{proof}
	Let $x$ be a solution of $(F^B)$, then obviously it is also a solution of $(F^B_T)$.\\
	Now let $x$ be a solution of $(F^B_T)$. By definition, $T(Ax-b)=0$, hence $y =Ax-b \in H$. Furthermore, $y$ belongs also to $S=-b+L_1$, therefore $y \in H \cap S$.
	Since $H$ avoids $S$, we deduce that $y=0$, i.e. $Ax=b$.
\end{proof}
 
\begin{cor}\label{cor:avoids}
		Let $H=\{y\in\mathbb{R}^m\ |\ Ty=0 \}$, with $T$ an $1\times m$ matrix, be a hyperplane that avoids the set $S'\supseteq -b+L_1$; then $T$ is a strong aggregation matrix for $(F^B)$.
\end{cor} 
\begin{proof}
	Obvious.
\end{proof}

In the rest of this section, we prove that we can build in polynomial time a hyperplane $H=\{y\in\mathbb{R}^m\ |\ Ty=0 \}$ that avoids the set $S=-b+L_1$.

It is well known that a basis $B\in \mathbb{Z}^{m \times m}$ of $\mathcal{L}(A)$ can be computed in polynomial time (cf., \cite{schrijver1998}). Hence we can check, in polynomial time,  whether $b$ belongs to $\mathcal{L}(A)$ or not. If $b \notin \mathcal{L}(A)$ then obviously $b \notin L_1$ and  $(F^B)$ is infeasible. Since problem $(F^B)$ is NP-hard that the lattice feasibility problem is polynomial,  we assume, for the sake of clarity that $b \in \mathcal{L}(A)$. 

\noindent Let us define the following matrix: $\tilde{A}=B^{-1}A$. By definition of $A$ and $B$, $\tilde{A}$ is a rational matrix. Furthermore all the lengths of the entries of $\tilde{A}$ are polynomial in the length of the inputs of $(F^B)$. Now let us define $M \in \mathbb{Z}_+$ large enough such that $\|\tilde{A}x\|_{\infty} \leq M $ for all $0 \leq x \leq u$. Again, notice that the length of $M$ is polynomial in the one of the inputs of $(F^B)$. Let us define the following bounded lattice: $L_2=\{By\ |\ \|y\|_{\infty}\leq M \}$.

\begin{lemme}\label{lem:0}
	We have $L_1 \subseteq L_2$.
\end{lemme}
\begin{proof}
	Let $Ax \in L_1$. By definition $Ax=B(\tilde{A}x)=By$ with $y=\tilde{A}x$. Since $0\leq x \leq u$ we have that $\|y\|\leq M$, hence $Ax \in L_2$.
\end{proof}

\noindent We now prove that if we can build a hyperplane $H=\{y\in\mathbb{R}^m\ |\ Ty=0 \}$ that avoids the set $-b+L_2$ then the feasibility problem $(F^B)$ is equivalent to its projected version $F^B_T$.

\noindent Let $\beta \in \mathbb{Z}^m$ such that $\beta = B^{-1}b$, we can rewrite 
$$-b+L_2=\{By\ |\ \forall i=1,...,m,\ -M-\beta_i\leq y_i\leq M-\beta_i\}$$ Let us define $C= M +\|\beta\|_{\infty}+1$. Let us consider the following integers $q_1,...,q_{m-1} > C$ that are pairwise coprime, i.e. $\gcd(q_i,q_j)=1$ for all $i\neq j$. We define the following matrix $\Lambda \in \mathbb{Q}^{(m-1)\times m}$: \\
\begin{equation}
\label{eq:lambda}
\Lambda =
\begin{pmatrix}
B_1 + \frac{1}{q_1}B_m \\
B_2 + \frac{1}{q_2}B_m \\
\vdots \\
B_{m-1} + \frac{1}{q_{m-1}}B_m
\end{pmatrix}
\end{equation}
 where $(B_1,...,B_m)=B$. Notice that $\mbox{rank}(\Lambda)=m-1$.

\begin{prop}\label{prop:proj}
	Let $t \in \mbox{Ker}(\Lambda)$. The hyperplane $H=\{y\in \mathbb{R}^m\ |\ t^\top y=0\}$ avoids $-b+L_2$.
\end{prop}
\begin{proof}
	Let $z \in H \cap (-b+L_2)$. The $m-1$ independent vectors \\
	$\{B_i+\frac{1}{q_i}B_m\}_{1\leq i\leq m-1}$, form a basis of $H$, hence there exists $(\lambda_1,...,\lambda_{m-1}) \in \mathbb{R}^{m-1}$ such that $z=\sum\limits_{i=1}^{m-1}\lambda_i(B_i+\frac{1}{q_i}B_m)=\sum\limits_{i=1}^{m-1}\lambda_iB_i+(\sum\limits_{i=1}^{m-1}\frac{\lambda_i}{q_i})B_m$. Since $ z\in-b+L_2=\{By\ |\ \forall i=1,...,m,\ -M-\beta_i\leq y_i\leq M-\beta_i\}$, there exists $y \in \mathbb{Z}^{m}$ such that $z=\sum_{i=1}^{m}y_iB_i,\ y\in \mathbb{Z}^{m},\ \|y\|_{\infty} < C$. Since $B$ is invertible, we deduce that for all $i=1,...,m-1$, $y_i=\lambda_i$, and $y_m=\sum\limits_{i=1}^{m-1}\frac{\lambda_i}{q_i}$. Hence $(\prod\limits_{i=1}^{m-1}q_i)y_m =\sum\limits_{i=1}^{m-1}(\prod\limits_{j\neq i}q_j )y_i$. For all $i\in \{1,...,m-1\}$, $q_i$ divides the left and side of the previous equation, therefore $q_i$ must divide $\sum\limits_{k=1}^{m-1}(\prod\limits_{j\neq k}q_j )y_k$. Hence we deduce that $q_i$ divides $(\prod\limits_{j\neq i}q_j )y_i$. Since $q_i$ and $\prod\limits_{j\neq i}q_j$ are coprime, we deduce that $q_i$ divides $y_i$. Since $q_i>C$ and $y_i\le C$, we conclude that $y_i=0$. 
	Hence $z=0$.
\end{proof}

\noindent By the proposition above, we deduce that if we can compute $q_1,...,q_{m-1}$ pairwise coprime then a rational projector $T\in \mathbb{Q}^{1 \times m}$ can be computed  by solving a linear system.\\

\noindent Let $\pi(x)$ be the {\it prime counting function}, i.e. $\pi(x)=|\{p \in \mathbb{N}\ |\ p\leq x,\mbox{ $p$ is prime}\}|$. The Prime Number Theorem asserts (\cite{bach1996}) that $\pi(x)\sim \frac{x}{\log(x)}$ and we can deduce that there exists two constants $c_1,c_2$ such that
$c_1\frac{x}{\log(x)} \leq \pi(x) \leq c_2\frac{x}{\log(x)}$. For all $i \in \mathbb{N}$, let $p_i$ be the $i$-th prime. A consequence of the Prime Number Theorem \cite{bach1996} is that for all $n\ge 6$, we have:
\begin{equation}\label{eq:prime}
p_n\le n(\log(n)+\log(\log(n)))
\end{equation}
We now prove  that we can build a strong aggregation matrix $T$ of size one in polynomial time

\begin{prop}\label{prop:projT}
We can build, in polynomial time, a strong aggregation matrix $T$ of size one for $(F^B)$.
\end{prop}
\begin{proof}
	By Proposition \ref{prop:proj}, all that remains to do is to prove that we can generate numbers, $q_1,...,q_{m-1} > C$ that are all pairwise coprime, in polynomial time.\\
	Let $K=(m-1)(\log(m-1)+\log(\log(m-1)))$, by (\ref{eq:prime}), we have $p_{m-1}\le K$. Using the sieve of Eratosthenes, we can generate all the prime numbers belonging to $\{2,...,K\}$  in $O(K\log(\log(K)))$. By definition of $K$, there exists at least $m-1$ of them: $p_1,...,p_{m-1}$. Now, for all $i=1,...,m-1$ let $\alpha_i$ be the smallest integer such that $p_i^{\alpha_i} >C$, Let $q_i=p_i^{\alpha_i}$. By definition, for all $i\neq j $, $q_i$ and $q_j$ are coprime. Furthermore, by definition of $p_i$ and $\alpha_i$, we have that $q_i\leq \max(K,C^2)$, therefore the size of $q_i$ is polynomial in the size of the input of the problem.
\end{proof}
Notice that by Corollary~\ref{cor:avoids}, it is enough to look for a hyperplane $H=\{y\in \mathbb{R}^m\ |\ Ty=0\}$, that avoids a set $S'$ containing $S=-b+L_1$. Hence we can consider a simpler lattice $L'$ that contains $\mathcal{L}(A)$ and that has a basis, $B'$, that is known. For example, let us consider the lattice, $\mathbb{Z}^m$, of integer vectors in $\mathbb{R}^m$. Let $M' \in \mathbb{N}$ be such that $\|Ax\|_\infty \le M'$ for all $x$ in $L_1$, and let $L_3=\{y \in \mathbb{Z}^m\, |\, \|y\|_\infty \le M' \}$.
We have $-b +L_1 \subset -b+L_3$.

\begin{cor}\label{cor:explicitT}
  Let $q_1,...,q_{m-1}$ be  pairwise coprime integers all greater than $M'+\|b\|_\infty$. Then, 
  $$T=\begin{pmatrix}
  \frac{1}{q_1} & \frac{1}{q_2} & \cdots & \frac{1}{q_{m-1}} & -1
  \end{pmatrix}$$
  is a strong aggregation matrix for $(F^B)$.
\end{cor}
\begin{proof}
	By proposition~\ref{prop:proj}, we just verify that $T^\top$ belongs to the null space of $\Lambda$ (c.f. equation~\ref{eq:lambda}), where $B$ is the identity matrix.
\end{proof}

\section{Strong aggregation for the general case}
\label{sec:2}
In this section we study, in the general case, the minimum size of a strong aggregation matrix for $(F)$. Let $\mathcal{C}(A)=\{Ax\, |\, x \ge 0\}$ be the polyhedral cone generated by $A$. We define $R=\mathcal{C}(A) \cap (- \mathcal{C}(A))$ the lineality subspace of $\mathcal{C}(A)$. We will prove that the minimum size of a strong aggregation matrix is equal to the dimension of the lineality subspace plus one. Furthermore, we will show that such a matrix can be constructed in polynomial time. This generalizes the existence result obtained in \cite{rosenberg1974}. We first consider the case where $\mathcal{C}(A)$ is pointed, i.e. $R=\{0\}$. Let us recall that $A$ is always assumed to be a full rank matrix.

\subsection{The pointed case}

We prove here that when the cone $\mathcal{C}(A)$ is pointed we can construct, in polynomial time a strong aggregation matrix for $(F)$. This generalizes the result obtained in \cite{kannan1983polynomial}, limited to the case where all the entries of $A$ are non-negatives. The idea is to first construct a hyperplane $H$ that will intersect the affine cone $-b +\mathcal{C}(A)$ in a bounded region. Then, following the ideas of the previous section, we will prove than we can ``perturb'' $H$ into a hyperplane that avoids $-b +\mathcal{L}^+(A)$. We first prove the following Lemma:

\begin{lemme}\label{lem:pointed}
	Assume that the cone $\mathcal{C}(A)$ is pointed, then there exists $h \in \mathbb{Z}^m$ such that $h^\top A \ge \|A\|_\infty + 1$. Furthermore, such $h$ can be computed in polynomial time.
\end{lemme}
\begin{proof}
	Let us consider the convex hull, $\mbox{conv}(A)=\{Ax\, |\, \sum\limits_{i=1}^nx_i =1,\, x\ge 0 \}$, of $A$. Since $\mathcal{C}(A)$ is pointed, $0 \notin \mbox{conv}(A)$, otherwise there would exist at least one column $A_i$ of $A$ that belongs to $R$. Hence, by the hyperplane separation theorem, there exists $d_1 < d_2 \in \mathbb{R}$ and $h \in \mathbb{R}^m$ such that $h^\top 0 \le d_1$ and $h^\top Ax \ge d_2$, for all $x \ge 0$. Therefore, $0\le d_1 < d_2$ and $h^\top A \ge d_2$. We conclude by scaling $h$ by some positive number in order that $h^\top A \ge \|A\|_\infty +1$. Notice that $h$ can be computed by solving a linear feasibility problem. Hence we can find $h \in \mathbb{Q}^m$, of polynomial size, in polynomial time. We conclude by scaling $h$ by a positive integer such that $h \in \mathbb{Z}$.
\end{proof}
By Lemma~\ref{lem:pointed}, let $h \in \mathbb{Z}^m$ such that $h^\top A \ge \|A\|_\infty +1$. We consider the hyperplane $H=\{y\in \mathbb{R}^m\, |\, h^\top y=0\}$. It it not difficult to see that the intersection $H\cap (-b + \mathcal{C}(A))$ is bounded (See Figure~\ref{fig:1}). In order to adapt the ideas developed in Section~\ref{sec:1}, we will now prove that we can perturb $h$ into $\tilde{h}=h+\eta$, such that the hyperplane $\tilde{H}=\{y\in \mathbb{R}^m\, |\, \tilde{h}^\top y=0\}$ intersects the affine cone $-b+\mathcal{C}(A)$ in a bounded region that contains no non-zero integer point, and hence avoids the set $-b +\mathcal{L}^+(A)$. In the Lemma bellow, we first prove that for any "small'' perturbation of $h$, the infinity norm of any element in  $\tilde{H}\cap (-b+\mathcal{C}(A))$ is bounded by a constant $M$.

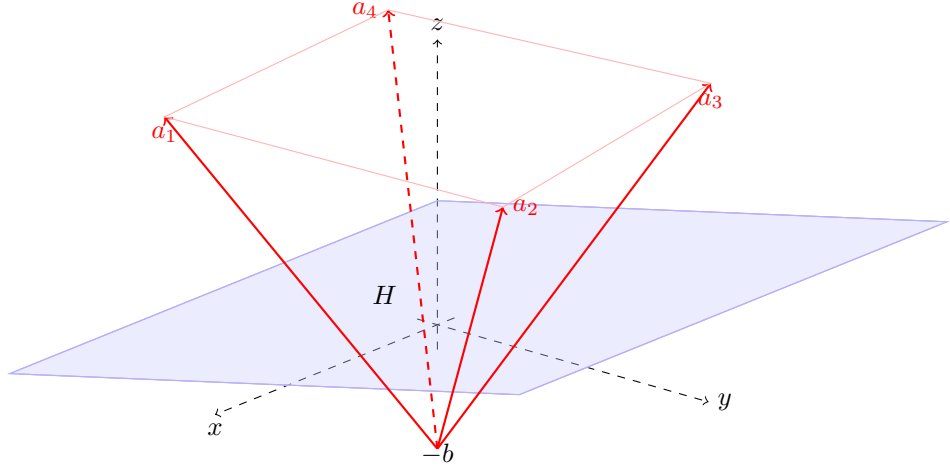
\begin{figure}
	\centering
	\begin{tikzpicture}[scale=3.5,tdplot_main_coords]
	\draw[dashed,->] (-0.1,0,0) -- (1.3,0,0) node[anchor=north]{$x$};
	\draw[dashed,->] (0,-0.1,0) -- (0,1.33,0) node[anchor=west]{$y$};
	\draw[dashed,->] (0,0,-0.1) -- (0,0,1.15) node[anchor=south]{$z$};
	\draw (0,0,-0.6) node[anchor=south] {$-b$};
	\node (p1) at (0,0,0.5){};
	\node (p2) at (2.5,0,0.5){};
	\node (p3) at (2.5,2.5,1){};
	\node (p4) at (0,2.5,1){};
	\draw[red!30,thin,fill=red!20] (p1) -- (p2) -- (p3) -- (p4) -- (p1);
	\draw[name path=plane 1,blue!30,thin,fill=blue!20,fill opacity=0.2] (0,0,0.5) -- (2.5,0,0.5) -- (2.5,2.5,1) -- (0,2.5,1) -- cycle;
	\filldraw[name path=plane 1,blue!30,thin,fill=blue!20,fill opacity=0.2] (0,0,0.5) -- (2.5,0,0.5) -- (2.5,2.5,1) -- (0,2.5,1) -- cycle;
	\draw (1.5,1,0.7) node[anchor=south] {$H$};
	
	\draw[name path=line 1,red,thick,->] (0.0,0,-0.5) -- (1,-0.5,1) node[anchor=north]{$a_1$};
	\draw[name path=line 2,red,thick,->] (0.0,0,-0.5) -- (2,2,1.5) node[anchor=west]{$a_2$};
	\draw[name path=line 3,red,thick,->] (0.0,0,-0.5) -- (1.5,2.6,2) node[anchor=north]{$a_3$};
	\draw[name path=line 4,red,dashed,thick,->] (0.0,0,-0.5) -- (-1.5,-1.5,0.5) node[anchor=east]{$a_4$};
	\draw[red!30,thin] (1,-0.5,1) -- (2,2,1.5) -- (1.5,2.6,2)  -- (-1.5,-1.5,0.5) -- cycle;
	
	\path [name intersections={of=plane 1 and line 1,by=A}];
	\path [name intersections={of=plane 1 and line 2,by=B}];
	\path [name intersections={of=plane 1 and line 3,by=C}];
	\path [name intersections={of=plane 1 and line 4,by=D}];

	\end{tikzpicture}
	\caption{Figure to illustrate Lemma~\ref{lem:pointed}} \label{fig:1}	
\end{figure}

\begin{lemme}\label{lem:boundedintersection}
	Let $h \in \mathbb{Z}^m$ such that $h^\top A \ge \|A\|_\infty +1$, we can compute $M>0$ such that, for all $\eta \in \mathbb{R}^m$, such that $\|\eta\|_\infty \le 1$, we have that  $$(-b+\mathcal{C}(A))\cap \{y\in \mathbb{R}^m\, |\, (h+\eta)^\top y=0\}$$
	 is bounded and any $y$ in the intersection verifies $\|y\|_\infty \le M$
\end{lemme}
\begin{proof}
	Notice first that by definition of $\|A\|_\infty$, we have that for all  $\|\eta\|_\infty \le 1$, $(h+\eta)^\top A \ge 1$. Hence, if $y \in (-b+\mathcal{C}(A))\cap \{y\in \mathbb{R}^m\, |\, (h+\eta)^\top y=0\}$, it implies that there exists $x \in \mathbb{Z}^n_+$ such that $(h+\eta)^\top Ax= (h+\eta)^\top b $. Hence, since $(h+\eta)^\top A \ge 1$, we have that $\|x\|_\infty \le (h+\eta)^\top b \le (\|h\|_\infty +1 )\|b\|_1$ (the last inequality is a consequence of Hölder inequality). Hence we can chose $M=\|A\|_\infty(\|h\|_\infty +1 )\|b\|_1 + \|b\|_\infty$, as $y=Ax-b$.
\end{proof}

 Let $q_1,...,q_{m}$ be  pairwise coprime integers all strictly greater than $M$, we now show how to compute a strong aggregation matrix for $(F)$.
 
 \begin{prop}\label{prop:pointedaggregate}
Let $h \in \mathbb{Z}^m$ such that $h^\top A \ge \|A\|_\infty +1$ and let $T=h^\top +\begin{pmatrix}
\frac{1}{q_1} & \frac{1}{q_2} & \cdots & \frac{1}{q_{m}}
\end{pmatrix}$. The hyperplane $\tilde{H}=\{y\in \mathbb{R}^m\, |\, Ty=0\}$ avoids $-b + \mathcal{L}^+(A)$.
 \end{prop}
\begin{proof}
	Let $z \in \tilde{H} \cap (-b+\mathcal{L}^+(A))$.
	Let $\kappa = -h^\top z$. Since $h$ and $z$ are both integer vectors, we have that 
	$$\sum\limits_{i=1}^m \frac{z_i}{q_i} =\kappa \in \mathbb{Z}$$
	Hence 	
		
 $$\sum\limits_{i=1}^{m}\left(\prod\limits_{j\neq i}q_j \right)z_i= \left(\prod\limits_{i=1}^mq_i \right)\kappa$$
 
 For all $i\in \{1,...,m\}$, $q_i$ divides the right-hand-side of  equation above, therefore $q_i$ must divide $\sum\limits_{k=1}^{m}(\prod\limits_{j\neq k}q_j )z_k$. Hence we deduce that $q_i$ divides $(\prod\limits_{j\neq i}q_j )z_i$. Since $q_i$ and $\prod\limits_{j\neq i}q_j$ are coprime, we deduce that $q_i$ divides $z_i$. Since $q_i>M$ and $z_i\le M$ (by Lemma~\ref{lem:boundedintersection}), we conclude that $z_i=0$. 	Hence $z=0$, which ends the proof.
\end{proof}

\begin{prop}\label{prop:onerowpointed}
	Assume that the cone $\mathcal{C}(A)$ is pointed, then we can construct, in polynomial time, a strong aggregation matrix $T$ for $(F)$
\end{prop}
\begin{proof}
	All that is left to do is to prove that we can compute $q_1,...,q_m > M$ in polynomial time. The proof of this is identical to the one of Proposition~\ref{prop:proj}.
\end{proof}

\subsection{The general case}

We assume now that the dimension, $r$, of the lineality subspace, $R$, of $\mathcal{C}(A)$ is greater than one. We first prove that any strong aggregation matrix for $(F)$ has a size of at least $r+1$.

\begin{lemme}\label{lem:LB}
	Let $r$ be the dimension of the lineality subspace, $R$, of $\mathcal{C}(A)$ and let $H=\{y\in \mathbb{R}^m\, |\, Ty=0\}$ be a subspace of dimension $m-k$ (where $T$ is a $k \times m$ rational matrix), with $k \le r$.
	Assume that $(F)$ has a solution $x^*$, then there exists a solution $x_T$ of $(F_T)$ that is not a solution of $(F)$.
\end{lemme}
\begin{proof}
	We first prove the following claim:\\
	 
	\noindent {\bf Claim:} There exists a non-zero vector $\tilde{y} \in \mathbb{Q}^m$  that belongs to $\in H\cap \mathcal{C}(A)$.\\
	Since $\dim(H) \ge m-r$, either: $H$ intersects $R$ in a non-zero vector $\tilde{y}$, which can be chosen rational since both $R$ and $H$ are generated by rational vectors; or $H$ and $R$ form a direct sum decomposition of $\mathbb{R}^m$ (i.e., $\mathbb{R}^m= H\oplus R$). Let us then consider $\hat{y}=A\hat{x} \in \mathcal{C}(A)\smallsetminus R$. Since $\mathbb{R}^m= H\oplus R$, there exists $v_1 \in H,\, v_2 \in R$, with $v_1 \neq 0$, such that $A\hat{x}=v_1+v_2$. Hence, since $v_2$ belongs to the lineality subspace of $\mathcal{C}(A)$, $\tilde{y}:=v_1=A\hat{x}-v_2 \in \mathcal{C}(A)$, and $\tilde{y}\neq 0$. Furthermore, for the same reason as above we can assume that $\tilde{y} \in \mathbb{Q}^m$, proving the claim.\\
	
	\noindent Let $\tilde{y} \in H\cap \mathcal{C}(A)$. There exists $\tilde{x} \in\mathbb{Q}^m_+$, such that $\tilde{y}=A\tilde{x}\neq 0$, and $TA\tilde{x}=0$. Since $\tilde{x}\in \mathbb{Q}^m_+$, there exists $\lambda >0$, such that $\lambda \tilde{x} \in \mathbb{Z}^m_+$. Let $x_T=x^*+ \lambda \tilde{x} \in \mathbb{Z}^m_+$, by definition of $\tilde{x}$, $Ax_T= b+\lambda \tilde{y} \neq b$ and $TAx_T=TAx^*+\lambda TA\tilde{x}=Tb$, this concludes the proof.
\end{proof}

Lemma~\ref{lem:LB}, proves that if $(F)$ is feasible, then a strong aggregation matrix, $T$, for $(F)$ has at least a size of $r+1$.

\begin{theo}\label{th:minsize}
	Let $r$ be the dimension of the lineality subspace $R$ of $\mathcal{C}(A)$, we can construct in polynomial time, an aggregation matrix $T$ of size $r+1$ for $(F)$. Furthermore, such an aggregation matrix is of minimum size.
\end{theo}
\begin{proof}
	The minimality is a consequence of Lemma~\ref{lem:LB}, we now exhibit such a aggregation matrix. The idea is to first construct a subspace, $H$ of dimension $m-r-1$ such that $H\cap (-b + \mathcal{C}(A))$ is bounded. Then, following the same idea as in the proof of Proposition~\ref{prop:pointedaggregate}, we will prove that we can perturb $H$ such that it avoids $-b+\mathcal{L}^+(A)$.\\
 Let $B_1,...,B_r \in \mathbb{Z}^m$ be a basis of the lattice $\mathcal{L}(A)\cap R$. Let $B_{r+1}=-\sum\limits_{i=1}^rB_i$, and define $B \in \mathbb{Z}^{m \times(r+1)}$ be the matrix that has its columns equal to the $B_i$ vectors for, $1\le i\le r+1$. Let $A^R$ (respectively $A^P$) be  the $m \times n_1$ (resp. $m \times n_2$) sub-matrix of $A$ that is made of all the columns $A_i$ that  belong to the lineality subspace $R$ (resp. that does not belong to $R$). Notice that the cone, $\mathcal{C}(A^P)$, generated by $A^P$ is pointed, otherwise the dimension of the lineality subspace of $\mathcal{C}(A)$ would be of at least $r$+1. Furthermore, since $R$ is a face of $\mathcal{C}(A)$, $\mathcal{C}(A^P) \cap R = \{0\}$.\\
	
	\noindent {\bf Claim 1:} For all $x\in \mathbb{Z}^n_+$ such that $Ax=b$, there exists $x^B \in \mathbb{Z}^{r+1}_+$ and $x^P \in \mathbb{Z}^{n_2}_+$ such that $Bx^B + A^Px^P=b$. \\
	 Let $x\in \mathbb{Z}^n_+$ such that $Ax=b$. We write $x=(x^R,x^P)$ and $Ax=A^Rx^R+A^Px^P$. Since $A^Rx^R \in \mathcal{L}(A)\cap R $, there exists $y \in \mathbb{Z}^r$ such that $A^Rx^R= \sum\limits_{i=1}^rB_iy_i$. By definition of $B_{r+1}=-\sum\limits_{i=1}^rB_i$, there exists $x^B \in \mathbb{Z}^{r+1}_+$ such that $\sum\limits_{i=1}^rB_iy_i= (\sum\limits_{i=1}^rB_ix^B_i)+x^B_{r+1}B_{r+1}$, proving the claim.\\
	 
	\noindent Let $\tilde{A}=\begin{pmatrix}
	 B & A^P
	\end{pmatrix}$, and let $\tilde{x}=\begin{pmatrix}
	x^B & x^P
	\end{pmatrix}^\top$. By Claim 1, if $H$ avoids $-b+\mathcal{L}^+(\tilde{A})$, then $H$ avoids $-b + \mathcal{L}^+(A)$. For all $k\in \{1,...,r+1\}$, we define $\tilde{A}^{\{k\}}$, as the matrix containing all the columns of $\tilde{A}$ except the $k$th. We now prove that for all  $k\in \{1,...,r+1\}$, the cone, $\mathcal{C}(\tilde{A}^{\{k\}})$, generated by $\tilde{A}^{\{k\}}$, is pointed.\\
	
	\noindent {\bf Claim 2:}  For all  $l\in \{1,...,r+1\}$, the cone, $\mathcal{C}(\tilde{A}^{\{l\}})$, generated by $\tilde{A}^{\{l\}}$, is pointed.\\
	Assume that there exists $l\in \{1,...,r+1\}$ such that $\tilde{A}^{\{l\}}$ is non-pointed. Hence, there exists a non-zero vector $(x^{\{l\}},x^P)\in \mathbb{R}^{r+n_2}_+$ such that $\sum\limits_{i\neq l}B_ix^{\{l\}}_i+A^Px^P=0$. Since the cone $\mathcal{C}(A^P)$ is pointed, we have $x^{\{l\}}\neq 0$. Furthermore, since $A^Px^P= -\sum\limits_{i\neq l}B_ix^{\{l\}}_i$, we must have that $x^P=0$, otherwise we would have that the pointed part and the non-pointed part of $\mathcal{C}(A)$ intersect in a non-zero element, which is impossible. Hence, $\sum\limits_{i\neq l}B_ix^{\{l\}}_i=0$, which is impossible since the family $(B_i)_{i\neq l}$ is linearly independent. Contradiction, the claim is proved.\\
	
	\noindent By Lemma~\ref{lem:pointed}, for all $l \in \{1,...,r+1\}$, we can compute, in polynomial time, an integer vector, $h_l \in \mathbb{Z}^m$ such that  $h_l^\top \tilde{A}^{\{l\}} \ge \|\tilde{A}^{\{l\}}\|_\infty + 1$. Let $\eta \in \mathbb{R}^m$ be any vector such that $\|\eta\|_\infty \le 1$. We define the following matrix: 
	$$T^\eta=\begin{pmatrix}
	(h_1+\eta)^\top  \\
	h_2^\top \\
	\vdots \\
	h_{r+1}^\top
	\end{pmatrix}= \begin{pmatrix}
	t_1 \\
	\vdots \\
	t_{r+1}
	\end{pmatrix}$$
	Notice that since $\|\eta\|_\infty \le 1$, $(h_1+\eta)^\top \tilde{A}^{\{1\}} \ge  1$. Let $H^\eta=\{y\in \mathbb{R}^m\, |\, T^\eta y=0\}$. We now prove that the intersection $H^\eta\cap (-b +\mathcal{C}(A))$ is bounded.\\
Let $y \in H^\eta \cap (-b +\mathcal{C}(A))$. Since $y \in -b +\mathcal{C}(A)$, we have  $y \in -b +\mathcal{C}(\tilde{A})$ (obviously, $\mathcal{C}(A)=\mathcal{C}(\tilde{A})$), hence there exists $\tilde{x}=(x^B,x^P)$ with $x^B \in \mathbb{R}^{r+1}_+$ and $x^P \in \mathbb{R}^{n_2}_+$ such that $y=-b+\tilde{A}\tilde{x}=-b+Bx^B + A^Px^P$. Furthermore, since $Bx^B=(\sum\limits_{i=1}^rB_ix^B_i)- x^B_{r+1}\sum\limits_{i=1}^rB_i$, we can chose $x^B$ such that at least one of its component, $x^B_l$, is equal to zero. Since $T^\eta y=0$, we have in particular that 
$$(T^\eta y)_l=t_l^\top y = t_l^\top(\tilde{A}\tilde{x}-b) =0$$
Since $x_l^B$, we have $t_l^\top\tilde{A}\tilde{x}=t_l^\top\tilde{A}^{\{l\}}\tilde{x}^{\{l\}}$, where $t_l^\top\tilde{A}^{\{l\}} \ge 1$, by definition of $T$.  Hence,  $t_l^\top\tilde{A}\tilde{x}=t_l^\top b$ implies that the infinity norm of $\tilde{x}$ is bounded by the infinity norm of $b$. This implies that $y$ is bounded and that $H^\eta\cap (-b +\mathcal{C}(A))$ is bounded. Let $M$ be an upper bound on the infinity norm of an element of $H^\eta\cap (-b +\mathcal{C}(A))$ (notice that  such an $M$ with polynomial size can be computed in polynomial time).\\

\noindent Let $q_1,...,q_{m}$ be  pairwise coprime integers all strictly greater than $M$, we fix now   $\eta$ to $\eta^* =\begin{pmatrix}
	\frac{1}{q_1} & \frac{1}{q_2} & \cdots & \frac{1}{q_{m}}
	\end{pmatrix}^\top$. Let $T=T^{\eta^*}$ and $H=H^{\eta^*}$. 	Let $y \in H \cap (-b+\mathcal{L}^+(A))$.
	Let $\kappa = -h_1\top y$. Since $h$ and $z$ are both integer vector, we have that 
	$$(Ty)_1=(h_1+\eta^*)^\top y =0 $$
	hence $$\sum\limits_{i=1}^m \frac{y_i}{q_i} =\kappa \in \mathbb{Z}$$
	hence 	
	$$\sum\limits_{i=1}^{m}\left(\prod\limits_{j\neq i}q_j \right)y_i= \left(\prod\limits_{i=1}^mq_i \right)\kappa$$
	
	For all $i\in \{1,...,m\}$, $q_i$ divides the right-hand-side of  equation above, therefore $q_i$ must divide $\sum\limits_{k=1}^{m}(\prod\limits_{j\neq k}q_j )y_k$. Hence we deduce that $q_i$ divides $(\prod\limits_{j\neq i}q_j )y_i$. Since $q_i$ and $\prod\limits_{j\neq i}q_j$ are coprime, we deduce that $y_i$ divides $y_i$. Since $q_i>M$ and $y_i\le M$, we conclude that $y_i=0$, and that $y=0$. We do not detail furthermore  the fact that $T$ can be computed in polynomial time as the proof is similar to the previous ones.
\end{proof}

\section{Weak aggregation}\label{sec:weakaggreg}

In the previous section, we have seen, in Theorem~\ref{th:minsize}, that the minimum size of a strong aggregation matrix is equal to the dimension of the lineality subspace of $\mathcal{C}(A)$ plus one. Hence, it is impossible, in the general case, to obtain an equivalent aggregated system of one linear equation. In this section, we will see that we can always compute, in polynomial time, a weak aggregated system $(\tilde{F})$ of size one. We recall that if $(\tilde{F})$ is a weak aggregated system for $(F)$ if the feasibility of $(F)$ is equivalent. The advantage of our approach, compared to \cite{kannan1983polynomial}, is that we do not introduced new variables in $\tilde{F}$, however the drawback is that we will introduce upper-bound for some of the variables of $(F)$.\\

Indeed, in \cite{borosh1976}, the authors prove that if $(F)$ is feasible, there exists a solution $x^*$ of $(F)$  that have an infinity norm bounded by a number $M$ that can be computed in polynomial time. Hence we can transform $(F)$ into a BILP, $(F^B)$, and use the framework developed in Section~\ref{sec:1} to compute a strong aggregated system $(F^B_T)$ of $(F^B)$, and hence a weak aggregated system for $(F)$. 

\section{Number of integer points in a polytope}
\label{sec:counting}
In this section, we derive a simple analytic formula to count the number of feasible solutions for $(F)$, in the case where the cone $\mathcal{C}(A)$ is pointed, or when explicit bounds on $x$ are known. In the general case, we will show that this formula can be used to know if $(F)$ is feasible or not. We first study the bounded case, and look for the number of feasible solutions for $(F^B)$:
\begin{equation*}
(F^B) = \left\{
\begin{array}{lll}
& Ax=b\\
& 0\leq x \leq u \\
& x \in \mathbb{Z}^n
\end{array}
\right. 
\end{equation*}
Let $T$ be a strong polynomial aggregation matrix for $(F_B)$, and let us rewrite $(F^B_T)$ as:
\begin{equation*}
(F^B_T) = \left\{
\begin{array}{lll}
& \sum\limits_{i=1}^n\alpha_i x_i =\beta\\
& 0\leq x \leq u \\
& x \in \mathbb{Z}^n
\end{array}
\right. 
\end{equation*}
Since $x$ is bounded, we can assume w.l.o.g. that $\alpha \ge 0$. Furthermore, since $\alpha$ is a rational vector, we also assume w.l.o.g. that $\alpha \in \mathbb{Z}^n$ and $\beta \in \mathbb{Z}$.
Let us consider the following polynomial:
$$P(X)=\prod\limits_{l=1}^n\bigg(\sum\limits_{j=0}^{u_l}X^{j\alpha_l}\bigg) $$ that can be rewritten into
$$P(X) = \sum\limits_{j=0}^\infty \gamma_jX^j$$
where $\gamma_i \in \mathbb{N}$ for all $i$.

\begin{lemme}
	\label{lem:cardinality}
	For all $i \in \mathbb{N}$, $\gamma_i$ is equal to the cardinality of the set $\left\{x \in \mathbb{Z}^n_+\, |\, \sum\limits_{j=1}^n \alpha_jx_j =i,\, x\le u \right\}$.
\end{lemme}
\begin{proof}
	Let  us expand the product $$P(X)=(\sum\limits_{j=0}^{u_1}X^{j\alpha_1})...(\sum\limits_{j=0}^{u_n}X^{j\alpha_n})=\sum\limits_{x\le u}X^{\sum\limits_{i=1}^n \alpha_i x_i}$$ 
	We conclude the proof rearranging the terms $X^i$ of the same degree.
\end{proof}
 
 Hence all we have to do, in order to compute the number of solutions of $(F_B)$, is to compute $\gamma_\beta$. Let $f:[-\pi,\pi]\mapsto \mathbb{C}$ be defined by $f(\theta)=P( e^{\imath \theta})$.

\begin{prop}\label{prop:solnumber1}
	The number, $\gamma_\beta$, of solutions of $(F^B)$ is given by the following expression:
	$$\gamma_\beta= \frac{1}{2\pi }\int_{-\pi}^\pi \left(
	\prod\limits_{j=1}^n \frac{1-e^{\imath  \alpha_j (u_j+1)\theta}}{1-e^{\imath  \alpha_j \theta}} \right) e^{-\imath  \beta \theta} d\theta $$
\end{prop}
\begin{proof}
	By definition,
	$$f(\theta)= \sum\limits_{j=0}^\infty \gamma_j e^{\imath j \theta} $$
	The set of functions $\{\theta \mapsto e^{\imath j \theta}\}_{j \in \mathbb{Z}}$ is an orthonormal basis of the space $L^2([-\pi, \pi ])$ of square integrable functions of $[-\pi,\pi]$, with respect to the scalar product
	$$\left\langle g, h\right\rangle= \frac{1}{2\pi }\int_{-\pi}^{\pi}g(\theta)\bar{h}(\theta)d\theta$$
	hence, since the sum in the expression of $f(\theta)$ is finite, we have:
	\begin{align*}
	& \frac{1}{2\pi }\int_{-\pi}^{\pi}f(\theta)e^{-\imath  \beta \theta} d\theta \\
	 = &\frac{1}{2\pi }\int_{-\pi}^{\pi}\left(\sum\limits_{j=0}^\infty \gamma_j e^{\imath j \theta}\right)e^{-\imath  \beta \theta} d\theta \\
	 = &\frac{1}{2\pi } \sum\limits_{j=0}^\infty \gamma_j \int_{-\pi}^{\pi} e^{\imath (j-\beta) \theta} \\
	 = & \gamma_\beta
	\end{align*}
	However, 
	\begin{align*}
	f(\theta)= & \prod\limits_{l=1}^n\bigg(\sum\limits_{j=0}^{u_l}e^{\imath j \alpha_l  \theta}\bigg) \\
	= & \prod\limits_{l=1}^n \frac{1-e^{\imath  \alpha_l (u_l+1)\theta}}{1-e^{\imath  \alpha_l \theta}}
	\end{align*}
	which ends the proof.
\end{proof}

Let us consider the unbounded case $(F)$ when $\mathcal{C}(A)$ is pointed. Similarly, let
$$\sum\limits_{i=1}^n\alpha_i x_i =\beta$$ be the aggregated equation. We can assume that all the coefficients are integer, furthermore we have seen in Proposition~\ref{prop:pointedaggregate}, that $\alpha > 0$, hence $\alpha \in \mathbb{Z}^n_+$. Let us consider the following formal polynomial:

$$Q(X)=\prod\limits_{l=1}^n\bigg(\sum\limits_{j=0}^{+\infty}X^{j\alpha_l}\bigg)= \sum\limits_{j=0}^\infty \sigma_jX^j $$
where $\sigma_i \in \mathbb{Z}_+$ for all $i$. Similarly, we have that the number of solutions of $(F)$ is given by the coefficient $\sigma_\beta$.

\begin{prop}\label{prop:solnumber2}
	The number, $\sigma_\beta$, of solutions of $(F)$ is given by the following expression:
	$$\sigma_\beta= \frac{2^{\beta-1}}{\pi }\int_{-\pi}^\pi \left(
	\prod\limits_{j=1}^n \frac{2^{\alpha_j}}{2^{\alpha_j}-e^{\imath  \alpha_j \theta}} \right) e^{-\imath  \beta \theta} d\theta $$
\end{prop}
\begin{proof}
	For all $|X| < 1$, $Q(X)$ converge and we have 
	$$Q(X)=\prod\limits_{l=1}^n \frac{1}{1-X^{\alpha_l}} $$
	Let $g:[-\pi,\pi]\mapsto \mathbb{C}$ be defined by $g(\theta)=Q( \frac{1}{2}e^{2\imath \pi \theta})$. We have
	$$g(\theta)= \prod\limits_{l=1}^n \frac{2^{\alpha_j}}{2^{\alpha_j}-e^{\imath  \alpha_j \theta}}$$
	We also have
	\begin{align*}
	g(\theta) &=  \sum\limits_{j=0}^\infty \sigma_j\left(\frac{1}{2}e^{\imath  \theta}\right)^j \\
	& = \sum\limits_{j=0}^\infty \frac{\sigma_j}{2^j}e^{\imath j \theta}
	\end{align*}
	Notice that since $\alpha \ge 0$, the above sum is finite. Hence by the same argument as in the proof of Proposition~\ref{prop:solnumber1}, we have that
	$$\frac{\sigma_\beta}{2^\beta}=\frac{1}{2\pi}\int_{-\pi}^{\pi}\left(\sum\limits_{j=0}^\infty \frac{\sigma_j}{2^j}e^{\imath j \theta}\right)e^{-\imath \beta \theta} d\theta $$
	which ends the proof.
\end{proof}

When no strong aggregation matrix of size one exists, it is no more possible to count the number of integer points in a polyhedron with this method. Indeed, in the non-pointed case, it is easy to see that if an integer solution point exists, then there is an infinity of integer solutions. However, by bounding artificially the solution space, as explained in the Section~\ref{sec:weakaggreg}, we can still know, using the formula of Proposition~\ref{prop:solnumber1}, if the problem is feasible or not.

\section{Bibliography}

\bibliographystyle{plain}
\bibliography{biblio}

\end{document}